\documentclass[10pt]{amsart}      
\usepackage{amssymb,amsthm, amsmath, amsfonts} 
\usepackage{graphics}                 
\usepackage{color}                    
\usepackage{hyperref}                 
\usepackage[all]{xy}
\usepackage{enumerate}

\newtheorem{theorem}{Theorem}[section]
\newtheorem{lemma}[theorem]{Lemma}
\newtheorem{proposition}[theorem]{Proposition}

\newtheorem{corollary}[theorem]{Corollary}
\newtheorem{definition}[theorem]{Definition}
\newtheorem{example}[theorem]{Example}

\numberwithin{equation}{section}

\DeclareMathOperator{\rad}{rad}
\DeclareMathOperator{\pd}{pd}
\DeclareMathOperator{\Hom}{Hom}
\DeclareMathOperator{\Ext}{Ext}
\DeclareMathOperator{\End}{End}
\DeclareMathOperator{\Ob}{Ob}

\DeclareMathOperator{\gldim}{gl.dim}
\DeclareMathOperator{\fdim}{fin.dim}
\DeclareMathOperator{\op}{op}

\title[A generalized Koszul theory]{A generalized Koszul theory and its relation to the classical theory}
\author{Liping Li}
\address{Department of Mathematics, University of California, Riverside, CA, 92507}
\email{lipingli@math.ucr.edu}
\begin{document}
\begin{abstract}
Let $A = \bigoplus _{i \geqslant 0} A_i$ be a graded locally finite $k$-algebra where $A_0$ is a finite-dimensional algebra whose finitistic dimension is 0. In this paper we develop a generalized Koszul theory preserving many classical results, and show an explicit correspondence between this generalized theory and the classical theory. Applications in representations of certain categories and extension algebras of standard modules of standardly stratified algebras are described.
\end{abstract}

\maketitle

\section{Introduction}

The classical Koszul theory plays an important role in the representation theory of graded algebras. However, there are a lot of structures (algebras, categories, etc) having natural gradings with non-semisimple degree 0 parts, to which the classical theory cannot apply. Particular examples of such structures include tensor algebras generated by non-semisimple algebras $A_0$ and $(A_0, A_0)$-bimodules $A_1$, extension algebras of finitely generated modules (among which we are most interested in extension algebras of standard modules of standardly stratified algebras \cite{Drozd, Li2}), graded modular skew group algebras, category algebras of finite EI categories \cite{Li4, Webb1, Webb2}, and certain graded $k$-linear categories. Therefore, it is reasonable to develop a generalized Koszul theory to study representations and homological properties of above structures.

In \cite{Green3,Madsen1,Madsen2,Woodcock} several generalized Koszul theories have been described, where the degree 0 part $A_0$ of a graded algebra $A$ is not required to be semisimple. In \cite{Woodcock}, $A$ is supposed to be both a left projective $A_0$-module and a right projective $A_0$-module. However, in many cases $A$ is indeed a left projective $A_0$-module, but not a right projective $A_0$-module. In Madsen's paper \cite{Madsen2}, $A_0$ is supposed to have finite global dimension. This requirement is too strong for us since in many applications $A_0$ is a self-injective algebra or a direct sum of local algebras, and hence $A_0$ has finite global dimension if and only if it is semisimple, falling into the framework of the classical theory. The theory developed by Green, Reiten and Solberg in \cite{Green3} works in a very general framework, and we want to find some conditions which are easy to check in practice. The author has already developed a generalized Koszul theory in \cite{Li1} under the assumption that $A_0$ is self-injective, and used it to study representations and homological properties of certain categories.

The goal of the work described in this paper is to loose the assumption that $A_0$ is self-injective (as required in \cite{Li1}) and replace it by a weaker condition so that the generalized theory can apply to more situations. Specifically, since we are interested in the extension algebras of modules, category algebras of finite EI categories, and graded $k$-linear categories for which the endomorphism algebra of each object is a finite dimensional local algebra, this weaker condition should be satisfied by self-injective algebras and finite dimensional local algebras. On the other hand, we also expect that many classical results as the Koszul duality can be preserved. Moreover, we hope to get a close relation between this generalized theory and the classical theory.

A trivial observation tells us that in the classical setup $A_0$ is semisimple if and only if $\gldim A_0$, the global dimension of $A_0$, is 0. Therefore, it is natural to consider the condition that $\fdim A_0$, the \textit{finitistic dimension} of $A_0$, is 0. Obviously, finite dimensional local algebras and self-injective algebras do have this property. It turns out that this weaker condition is suitable for our applications, and many classical results still hold.

Explicitly, let $A = \bigoplus _{i \geqslant 0} A_i$ be a graded \textit{locally finite} $k$-algebra generated in degrees 0 and 1, i.e., $\dim _k A_i < \infty$ and $A_1 \cdot A_i = A_{i+1}$ for all $i \geqslant 0$. We assume that both $\fdim A_0$ and $\fdim A_0^{\op}$ are 0, where $A_0^{\op}$ is the opposite algebra of $A_0$. We then define \textit{generalized Koszul modules} and \textit{generalized Koszul algebras} by linear projective resolutions, as people did for classical Koszul modules and classical Koszul algebras.

It is well known that in the classical theory \textit{linear modules} (defined by linear projective resolutions) and \textit{Koszul modules} (defined by a certain extension property) coincide. We have a similar result:

\begin{theorem}
Let $A = \bigoplus _{i \geqslant 0} A_i$ be a locally finite graded algebra with $\fdim A_0 = \fdim A_0^{\op} = 0$. If $A$ is a projective $A_0$-module, then a graded module $M$ is generalized Koszul if and only if it is a projective $A_0$-module and the graded $\Gamma = \Ext _A^{\ast} (A_0, A_0)$-module $\Ext _A^{\ast} (M, A_0)$ is generated in degree 0, i.e.,
\begin{equation*}
\Ext _{A}^1 (A_0, A_0) \cdot \Ext _{A}^i (M, A_0)=  \Ext _{A}^{i+1} (M, A_0).
\end{equation*}
\end{theorem}

We also have the generalized Koszul duality as follows:

\begin{theorem}
Let $A = \bigoplus _{i \geqslant 1} A_i$ be a locally finite graded algebra with $\fdim A_0 = \fdim A_0^{\op} = 0$. If $A$ is a generalized Koszul algebra, then $E = \Ext ^{\ast}_{A} (-, A_0)$ gives a duality between the category of generalized Koszul $A$-modules and the category of generalized Koszul $\Gamma = \Ext _A^{\ast} (A_0, A_0)$-modules. That is, if $M$ is a generalized Koszul $A$-module, then $E(M)$ is a generalized Koszul $\Gamma$-module, and $E_{\Gamma}EM = \Ext ^{\ast} _{\Gamma} (EM, \Gamma_0) \cong M$ as graded $A$-modules.
\end{theorem}

Let $\mathfrak{r}$ be the radical of $A_0$ and define $\mathfrak{R} = A \mathfrak{r} A$ to be the two-sided ideal generated by $\mathfrak{r}$.  For a graded $A$-module $M = \bigoplus _{i \geqslant 0} M_i$, we then define a quotient algebra $\bar{A} = A / A \mathfrak{r} A = \bigoplus _{i \geqslant 0} A_i / (A \mathfrak{r} A)_i$ and $\bar{M} = M / \mathfrak{R} M =  \bigoplus _{i \geqslant 0} M_i / (\mathfrak{R} M)_i$. Clearly, $\bar{M}$ is a graded $\bar{A}$-module, and the graded $A$-module $M$ is generated in degree 0 if and only if the corresponding graded $\bar{A}$-module $\bar{M}$ is generated in degree 0. Moreover, in the situation that $\mathfrak{r}A_1 = A_1 \mathfrak{r}$, we get the following correspondence between our generalized Koszul theory and the classical theory:

\begin{theorem}
Let $A = \bigoplus _{i \geqslant 1} A_i$ be a locally finite graded algebra with $\fdim A_0 = \fdim A_0^{\op} = 0$ and suppose $\mathfrak{r}A_1 = A_1\mathfrak{r}$. Then:
\begin{enumerate}
\item $A$ is a generalized Koszul algebra if and only it is a projective $A_0$-module and $\bar{A}$ is a classical Koszul algebra.
\item Suppose that $A$ is a projective $A_0$-module. A graded $A$-module $M$ is generalized Koszul if and only if it is a projective $A_0$-module and the corresponding grade $\bar{A}$-module $\bar{M}$ is classical Koszul.
\end{enumerate}
\end{theorem}

We then apply the generalized Koszul theory to some finite categories. They play an important role in the theory of finite groups and their representations. Except the correspondence established in the previous theorem, we give other correspondences in Section 4 (see Theorems 4.1 and 4.2), showing that the generalized Koszul property of the category algebras of such categories is equivalent to the classical Koszul property of quotient algebras of certain finite dimensional hereditary algebras. In practice, the latter one is much easier to check since it is not hard to construct graded projective resolutions for simple modules.

Extension algebras of standard modules of standardly stratified algebras have been widely studied in \cite{Abe, Agoston1, Agoston2, Drozd, Klamt, Li2, Madsen3, Mazorchuk1, Mazorchuk3, Mazorchuk4, Miemietz}. Different from the usual approach, we do not assume that the standardly stratified algebra is graded, the degree 0 part is semisimple, and its grading is compatible with the filtration by standard modules. Instead, using a combinatorial property of the filtration by standard modules, we show that the extension algebra of standard modules has generalized Koszul property.

The paper is organized as follows: In the next section we develop the generalized Koszul theory and prove the first two theorems. In Section 3 we describe the relation between the generalized theory and the classical theory, and prove the third theorem. Applications of this theory and the correspondence to certain categories are described in Section 4. In the last section we discuss the generalized Koszul property of extension algebras of standard modules.

We introduce some notation here. Throughout this paper $k$ is an algebraically closed field. Let $A = \bigoplus _{i \geqslant 0} A_i$ be a \textit{locally finite} graded algebra generated in degrees 0 and 1, i.e., $\dim _k A_i < \infty$ and $A_{i+1} = A_1 \cdot A_i$ for all $i \geqslant 0$. An $A$-module $M = \bigoplus _{i \geqslant 0} M_i$ is \textit{graded} if $A_i \cdot M_j \subseteq M_{i+j}$. It is said to be \textit{generated in degree $s$} if $M = A \cdot M_s$. It is \textit{locally finite} if $\dim _k M_i < \infty$ for all $i \geqslant 0$. In this paper all graded modules are supposed to be locally finite.

Given two graded $A$-modules $M$ and $N$, $\Hom_A (M, N)$ and $\hom_A (M, N)$ are the spaces of all module homomorphisms and of all graded module homomorphisms respectively. The composition of maps $f: L \rightarrow M$ and $g: M \rightarrow N$ is denoted by $gf$. The degree shift functor $[-]$ is defined by letting $M[i]_s = M_{s-i}$ for $i, s \in \mathbb{Z}$. Denote $\mathfrak{J} = \bigoplus _{i \geqslant 1} A_i$, which is a two-sided ideal of $A$. We identify $A_0$ with the quotient module $A/ \mathfrak{J}$ and view it as a graded $A$-module concentrated in degree 0. We view the zero module 0 as a projective module since this can simplify the expression of some statement.

\section{A generalized Koszul theory}

We start with some preliminary results, most of which are generalized from those described in \cite{BGS, Green1, Green2, Martinez, Mazorchuk2}. The reader is also suggested to look at other generalized Koszul theories described in \cite{Green3, Madsen1, Madsen2, Woodcock}.

The following lemmas are proved in \cite{Li1}, where we did not use the condition that $A_0$ is self-injective (Remark 2.8 in \cite{Li1}).

\begin{lemma}
(Lemma 2.1 in \cite{Li1}) Let $A$ be as above and $M$ be a graded $A$-module. Then:
\begin{enumerate}
\item $\mathfrak{J}$ is contained in the graded radical of $A$;
\item $M$ has a graded projective cover;
\item the graded syzygy $\Omega M$ is also locally finite.
\end{enumerate}
\end{lemma}

\begin{lemma}
(Lemma 2.2 in \cite{Li1}) Let $0 \rightarrow L \rightarrow M \rightarrow N \rightarrow$ be an exact sequence of graded $A$-modules. Then:
\begin{enumerate}
\item If $M$ is generated in degree $s$, so is $N$.
\item If $L$ and $N$ are generated in degree $s$, so is $M$.
\item If $M$ is generated in degree $s$, then $L$ is generated in degree $s$ if and only if $\mathfrak{J}M \cap L = \mathfrak{J}L$.
\end{enumerate}
\end{lemma}

Now we define \textit{generalized Koszul modules} and \textit{generalized Koszul algebras}.

\begin{definition}
A graded $A$-module $M$ is called a generalized Koszul module if it has a (minimal) linear projective resolution
\begin{equation*}
\xymatrix{ \ldots \ar[r] & P^n \ar[r] & P^{n-1} \ar[r] & \ldots \ar[r] & P^0 \ar[r] & M \ar[r] & 0}
\end{equation*}
such that $P^i$ is generated in degree $i$ for all $i \geqslant 0$. The graded algebra $A$ is called a generalized Koszul algebra if $A_0$ viewed as an $A$-module is generalized Koszul.
\end{definition}

The reader can easily see that $M$ is a generalized Koszul $A$-module if and only if $M$ is generated in degree 0 and each syzygy $\Omega^i (M)$ is generated in degree $i$ for every $i \geqslant 1$. Moreover, from the above projective resolution, we deduce that $M_0 \cong P^0_0$ and $\Omega^i (M)_i \cong P^i_i$ are projective $A_0$-modules for all $i \geqslant 1$.

Recall for a finite dimensional algebra $\Lambda$, the \textit{finitistic dimension} $\fdim \Lambda$ is defined as the supremum of projective dimensions of all indecomposable $\Lambda$-modules having finite projective dimension (\cite{Bass, Jans}). In particular, if the global dimension $\gldim \Lambda$ is finite, then $\fdim \Lambda = \gldim \Lambda$. It is well known that $\fdim \Lambda = 0$ if $\Lambda$ is a finite dimensional local algebras or a self-injective algebra. The famous finitistic dimension conjecture asserts that the finitistic dimension of any finite dimensional algebra is finite.

From now on we assume that $\fdim A_0 = 0 = \fdim A_0^{\op}$. It is easy to see that this assumption is equivalent to the following splitting condition:\\

\textbf{(S): Every exact sequence $0 \rightarrow P \rightarrow Q \rightarrow R \rightarrow 0$ of left (right, resp.) $\Lambda$-modules splits if $P$ and $Q$ are left (right, resp.) projective $\Lambda$-modules.}\\

Indeed, from the short exact sequence we deduce that $\pd _{\Lambda} R \leqslant 1$. If $\fdim \Lambda = 0 = \fdim \Lambda ^{\op}$, then $\pd _{\Lambda} R = 0$, and hence the exact sequence splits. Conversely, suppose that every such exact sequence splits. If there is some $\Lambda$-module $M$ such that $\pd _{\Lambda} M = n \geqslant 1$, we consider $R = \Omega^{n-1} (M)$ and deduce that $\pd _{\Lambda} R = 1$. Consequently, the projective resolution of $R$ is non-splitting, contradicting the splitting condition.

\begin{proposition}
Let $0 \rightarrow L \rightarrow M \rightarrow N \rightarrow$ be an exact sequence of graded $A$-modules such that $L$ is generalized Koszul. Then $M$ is generalized Koszul if and only if $N$ is generalized Koszul.
\end{proposition}

\begin{proof}
This is Proposition 2.9 in \cite{Li1}. The proof is almost the same except replacing the self-injective property of $A_0$ by the splitting property (S). For the sake of completeness we give a brief proof here.

By the second statement of the previous lemma, $M$ is generated in degree 0 if and only if $N$ is generated in degree 0. Consider the following diagram in which all rows and columns are exact:
\begin{align*}
\xymatrix{
& 0 \ar[d] & 0 \ar[d] & 0 \ar[d] &  \\
0 \ar[r] & \Omega L \ar[r] \ar[d] & M' \ar[r] \ar[d] & \Omega N \ar[r] \ar[d] & 0\\
0 \ar[r] & P \ar[r] \ar[d] & P \oplus Q \ar[r] \ar[d] & Q \ar[r] \ar[d] & 0 \\
0 \ar[r] & L \ar[r] \ar[d] & M \ar[r] \ar[d] & N \ar[r] \ar[d] & 0 \\
& 0 & 0 & 0. &
}
\end{align*}
Here $P$ and $Q$ are graded projective covers of $L$ and $N$ respectively. We claim $M' \cong \Omega M$. Indeed, the given exact sequence induces an exact sequence of $A_0$-modules:
\begin{align*}
\xymatrix { 0 \ar[r] & L_0 \ar[r] & M_0 \ar[r] & N_0 \ar[r] & 0.}
\end{align*}
Observe that $L_0$ is a projective $A_0$-module. If $N$ is generalized Koszul, then $N_0$ is a projective $A_0$-module since $N_0 \cong Q^0_0$, and the above sequence splits. If $M$ is generalized Koszul, then $M_0$ is a projective $A_0$-module, and this sequence splits as well by the splitting property (S). In either case we have $M_0 \cong L_0 \oplus N_0$. Thus $P \oplus Q$ is a graded projective cover of $M$, and hence $M' \cong \Omega M$ is generated in degree 1 if and only if $\Omega N$ is generated in degree 1 by Lemma 2.2. Replace $L$, $M$ and $N$ by $(\Omega L)[-1]$, $(\Omega M)[-1]$ and $(\Omega N)[-1]$ (all of them are generalized Koszul) respectively in the short exact sequence. Repeating the above procedure we prove the conclusion by recursion.
\end{proof}

If $M$ is a generalized Koszul module, its truncations (with suitable degree shifts) are generalized Koszul as well:

\begin{proposition}
Let $A$ be a generalized Koszul algebra and $M$ be a generalized Koszul module. Then $\mathfrak{J}^i M[-i]$ is also generalized Koszul for each $i \geqslant 1$.
\end{proposition}

\begin{proof}
This is Proposition 2.13 in \cite{Li1}. For the convenience of the reader we include a brief proof here. Consider the following commutative diagram:
\begin{align*}
\xymatrix{
& 0 \ar[r] & \Omega M \ar[r] \ar[d] & \Omega(M_0) \ar[r] \ar[d] & \mathfrak{J}M \ar[r] & 0 \\
& 0 \ar[r] & P^0 \ar[r]^{id} \ar[d] & P^0 \ar[r] \ar[d] & 0 & \\
0 \ar[r] & \mathfrak{J} M \ar[r] & M \ar[r] & M_0 \ar[r] & 0 &}
\end{align*}
Since $M_0$ is a projective $A_0$-module and $A_0$ is generalized Koszul, $\Omega(M_0)[-1]$ is also generalized Koszul. Similarly, $\Omega M[-1]$ is generalized Koszul since so is $M$. Therefore, $\mathfrak{J} M[-1]$ is generalized Koszul by the previous proposition. Now replacing $M$ by $\mathfrak{J} M[-1]$ and using recursion, we conclude that $\mathfrak{J} ^iM[-i]$ is a generalized Koszul $A$-module for every $i \geqslant 1$.
\end{proof}

From this proposition we immediately deduce that if $A$ is a generalized Koszul algebra, then it is a projective $A_0$-module. We now focus on graded algebras with this property.

\begin{proposition}
If $A$ is a projective $A_0$-module, then every generalized Koszul module $M$ is a projective $A_0$-module.
\end{proposition}

\begin{proof}
Clearly, it suffices to show that $M_i$ is a projective $A_0$-module for each $i \geqslant 0$. Since $M$ is generalized Koszul, $M_0$ is a projective $A_0$-module. Now suppose $i \geqslant 1$. The minimal linear projective resolution of $M$ gives rise to exact sequences of $A_0$-modules:
\begin{equation*}
\xymatrix{0 \ar[r] & \Omega^{s+1} (M)_i \ar[r] & P^s_i \ar[r] & \Omega^s (M)_i \ar[r] & 0,} \quad 0 \leqslant s \leqslant i.
\end{equation*}
If $s = i$, we have $\Omega^{i+1} (M)_i = 0$ since $\Omega^{i+1} (M)$ is generated in degree $i+1$. Thus $\Omega^i (M)_i \cong P^i_i$ is a projective $A_0$-module. Now let $s = i-1$. We claim that the first term $\Omega^i (M)_i$ is a projective $A_0$-module. Indeed, $\Omega^i (M) [-i]$ is a generalized Koszul module, so $(\Omega^i (M) [-i])_0$ is a projective $A_0$-module. But $\Omega^i (M)_i \cong (\Omega^i (M) [-i])_0$. This proves the claim. Since the first two terms are projective $A_0$-modules, by the splitting property (S), we deduce that $\Omega ^{i-1} (M)_i$ is a projective $A_0$-module. By recursion, we conclude that $M_i$ is a projective $A_0$-module for every $i >0$.
\end{proof}

The following lemma will be used in the proof of Theorem 1.1.

\begin{lemma}
Let $M$ be a graded $A$-module generated in degree 0. Suppose that both $A$ and $M$ are projective $A_0$-modules. Then $\Omega M$ is generated in degree 1 if and only if every $A$-module homomorphism $\Omega M \rightarrow A_0$ extends to an $A$-module homomorphism $\mathfrak{J} P \rightarrow A_0$, where $P$ is a graded projective cover of $M$.
\end{lemma}

\begin{proof}
This is a varied version of Lemma 2.17 in \cite{Li1}. The exact sequence $0 \rightarrow \Omega M \rightarrow P \rightarrow M \rightarrow 0$ induces an exact sequence $
0 \rightarrow (\Omega M)_1 \rightarrow P_1 \rightarrow M_1 \rightarrow 0$ of $A_0$-modules, which splits since $M_1$ is a projective $A_0$-module. Applying the functor $\Hom_{A_0} (-, A_0)$ we get another split exact sequence
\begin{equation*}
0 \rightarrow \Hom _{A_0} (M_1, A_0) \rightarrow \Hom _{A_0} (P_1, A_0) \rightarrow \Hom _{A_0} ((\Omega M)_1, A_0) \rightarrow 0.
\end{equation*}
Note that $(\Omega M)_0 =0$. Therefore, $\Omega M$ is generated in degree 1 if and only if $\Omega M / \mathfrak{J} (\Omega M) \cong (\Omega M)_1$, if and only if the above sequence is isomorphic to
\begin{equation*}
0 \rightarrow \Hom _{A_0} (M_1, A_0) \rightarrow \Hom _{A_0} (P_1, A_0) \rightarrow \Hom _{A_0} (\Omega M / \mathfrak{J} \Omega M, A_0) \rightarrow 0.
\end{equation*}
Here we use the fact that $M_1$, $P_1$ and $(\Omega M)_1$ are projective $A_0$-modules. But the above sequence is isomorphic to
\begin{equation*}
0 \rightarrow \Hom _A (\mathfrak{J}M, A_0) \rightarrow \Hom _A (\mathfrak{J} P, A_0) \rightarrow \Hom _A (\Omega M, A_0) \rightarrow 0
\end{equation*}
since $\mathfrak{J}M$ and $\mathfrak{J} P$ are generated in degree 1. Therefore, $\Omega M$ is generated in degree 1 if and only if every (non-graded) $A$-module homomorphism $\Omega M \rightarrow A_0$ extends to a (non-graded) $A$-module homomorphism $\mathfrak{J} P \rightarrow A_0$.
\end{proof}

Let $M$ be a graded $A$-module and $\Gamma = \Ext _A^{\ast} (A_0, A_0)$. Then $\Ext _A^{\ast} (M, A_0)$ is a graded $\Gamma$-module. Now we restate and prove Theorem 1.1.

\begin{theorem}
Let $A = \bigoplus _{i \geqslant 0} A_i$ be a locally finite graded algebra with $\fdim A_0 = \fdim A_0^{\op} = 0$. If $A$ is a projective $A_0$-module, then a graded module $M$ is generalized Koszul if and only if it is a projective $A_0$-module and the graded $\Gamma = \Ext _A^{\ast} (A_0, A_0)$-module $\Ext _A^{\ast} (M, A_0)$ is generated in degree 0, i.e.,
\begin{equation*}
\Ext _{A}^1 (A_0, A_0) \cdot \Ext _{A}^i (M, A_0)=  \Ext _{A}^{i+1} (M, A_0).
\end{equation*}
\end{theorem}

\begin{proof}
This is a varied version of Theorem 2.16 in \cite{Li1}. Since the proof is almost the same, we only give a sketch. Please refer to \cite{Li1} for details.

\textbf{The only if part.} Let $M$ be a generalized Koszul $A$-module. Without loss of generality we can suppose that $M$ is indecomposable. By Lemma 2.6 $M$ is a projective $A_0$-module. As in the original proof, it suffices to show that the given identity is true for $i = 1$, i.e.,
\begin{equation*}
\Ext _{A}^1 (M, A_0) = \Ext _{A}^1 (A_0, A_0) \cdot \Hom _{A} (M, A_0).
\end{equation*}
The proof of this identity is completely the same as the original proof. We omit the details.

\textbf{The if part.} As in the original proof, we only need to show that $\Omega M$ is generated in degree 1. By the previous lemma, it suffices to show that each (non-graded) $A$-module homomorphism $g: \Omega M \rightarrow A_0$ extends to $\mathfrak{J} P^0$, where $P^0$ is a graded projective cover of $M$. The proof of this fact is completely the same as the original proof.
\end{proof}

An immediate corollary of the above theorem is:

\begin{corollary}
The graded algebra $A$ is generalized Koszul if and only if $A$ is a projective $A_0$-module and $\Gamma = \Ext _A^{\ast} (A_0, A_0)$ is generated in degrees 0 and 1.
\end{corollary}

Now we can prove a generalized Koszul duality.

\begin{theorem}
Let $A = \bigoplus _{i \geqslant 1} A_i$ be a locally finite graded algebra with $\fdim A_0 = \fdim A_0^{\op} = 0$. If $A$ is a generalized Koszul algebra, then $E = \Ext ^{\ast}_{A} (-, A_0)$ gives a duality between the category of generalized Koszul $A$-modules and the category of generalized Koszul $\Gamma = \Ext _A^{\ast} (A_0, A_0)$-modules. That is, if $M$ is a generalized Koszul $A$-module, then $E(M)$ is a generalized Koszul $\Gamma$-module, and $E_{\Gamma}EM = \Ext ^{\ast} _{\Gamma} (EM, \Gamma_0) \cong M$ as graded $A$-modules.
\end{theorem}

\begin{proof}
This is a varied version of Theorem 4.1 in \cite{Li1}. We give a detailed proof for the convenience of the reader. Since $A_0$ is a generalized Koszul module and $M$ is a projective $A_0$-module, $M_0$ is generalized Koszul as well. By Proposition 2.5, $\mathfrak{J} M[-1]$ is also generalized Koszul. Furthermore, we have the following short exact sequence of generalized Koszul modules:
\begin{align*}
\xymatrix{ 0 \ar[r] & \Omega M [-1] \ar[r] & \Omega (M_0)[-1] \ar[r] & \mathfrak{J} M[-1] \ar[r] & 0.}
\end{align*}
As in the proof of Proposition 2.4, this sequence induces exact sequences of generalized Koszul modules recursively:
\begin{align*}
\xymatrix{ 0 \ar[r] & \Omega^i(M)[-i] \ar[r] & \Omega^i(M_0)[-i] \ar[r] & \Omega^{i-1}(\mathfrak{J} M[-1]) [1-i] \ar[r] & 0.}
\end{align*}
Take a fixed sequence for a certain $i > 0$. It gives a splitting exact sequence of $A_0$-modules:
\begin{align*}
\xymatrix{ 0 \ar[r] & \Omega^i(M)_i \ar[r] & \Omega^i(M_0)_i \ar[r] & \Omega^{i-1}(\mathfrak{J}M [-1])_{i-1} \ar[r] & 0.}
\end{align*}
Applying $\Hom_{A_0} (-, A_0)$ to it and using the following isomorphism for a graded $A$-module $N$ generated in degree $i$
\begin{equation*}
\Hom _{A} (N, A_0) \cong \Hom_{A} (N_i, A_0) \cong \Hom _{A_0} (N_i, A_0,)
\end{equation*}
we get:
\begin{align*}
0 \rightarrow \Hom _{A} (\Omega^{i-1}(\mathfrak{J} M[-1]), A_0) \rightarrow \Hom _{A} (\Omega^i(M_0), A_0) \rightarrow \Hom _{A} (\Omega^iM, A_0) \rightarrow 0,
\end{align*}
which is isomorphic to
\begin{align*}
0 \rightarrow \Ext^{i-1} _{A} (\mathfrak{J} M[-1], A_0) \rightarrow \Ext^i _A (M_0, A_0) \rightarrow \Ext^i _A (M, A_0) \rightarrow 0.
\end{align*}
Now let the index $i$ vary and put these sequences together. We have:
\begin{align*}
\xymatrix{ 0 \ar[r] & E(\mathfrak{J} M[-1])[1] \ar[r] & E(M_0) \ar[r] & EM \ar[r] & 0.}
\end{align*}

Let us focus on this sequence. We claim $\Omega(EM) \cong E(\mathfrak{J} M[-1])[1]$. Indeed, since $M_0$ is a projective $A_0$-module, $E(M_0)$ is a projective $\Gamma$-module. But $\mathfrak{J} M[-1]$ is generalized Koszul, so $E(\mathfrak{J} M[-1])$ is generated in degree 0 by the previous theorem. Thus $E(\mathfrak{J} M[-1])[1]$ is generated in degree 1, and $E(M_0)$ is a graded projective cover of $EM$. This proves the claim. Consequently, $\Omega(EM)$ is generated in degree 1. Moreover, replacing $M$ by $\mathfrak{J} M[-1]$ (which is also generalized Koszul) and using the claimed identity, we have that
\begin{equation*}
\Omega^2(EM) = \Omega(E(\mathfrak{J} M[-1])[1]) = \Omega( E(\mathfrak{J} M[-1]) [1] = E( \mathfrak{J} ^2M[-2])[2],
\end{equation*}
is generated in degree 2. By recursion, $\Omega^i (EM) \cong E(\mathfrak{J} ^iM [-i])[i]$ is generated in degree $i$ for all $i \geqslant 0$. Thus $EM$ is a generalized Koszul $\Gamma$-module (note that $\Gamma_0 \cong A_0^{\op}$ and $\fdim \Gamma_0 = \fdim A_0^{\op} = 0$). In particular for $M = _AA$,
\begin{equation*}
EA = \Ext _{A} ^{\ast} (A, A_0) = \Hom _A (A, A_0) = \Gamma_0
\end{equation*}
is a generalized Koszul $\Gamma$-module.

Since $\Omega ^i (EM)$ is generated in degree $i$,
\begin{align*}
\Omega^i (EM)_i & \cong E(\mathfrak{J} ^iM [-i])[i]_i \cong E(\mathfrak{J} ^iM[-i])_0 \\
& = \Hom _{A} (\mathfrak{J} ^iM[-i], A_0) \cong \Hom _{A} (M_i, A_0).
\end{align*}
We also have
\begin{align*}
\Hom_{\Gamma} (\Omega^i(EM), \Gamma_0) & \cong \Hom _{\Gamma _0} (\Omega^i(EM)_i, \Gamma_0) \nonumber \\
& \cong \Hom _{\Gamma _0} (\Hom_{A} (M_i, A_0), \Gamma_0) \nonumber \\
& \cong \Hom _{\Gamma_0} (\Hom_{A_0} (M_i, A_0), \Gamma_0) \nonumber \\
& \cong M_i.
\end{align*}
The last isomorphism holds because $M_i$ is a projective $A_0$-module and $\Gamma_0 \cong A_0^{op}$. Therefore, we get
\begin{equation*}
\Ext _{\Gamma}^i (EM, \Gamma_0) \cong \Hom_{\Gamma} (\Omega^i(EM), \Gamma_0) \cong M_i
\end{equation*}
for every $i \geqslant 0$. Adding them together, $E_{\Gamma}E(M) \cong \bigoplus_{i=0}^{\infty} M_i \cong M$.\

Now we have $E_{\Gamma} ( E(A)) = E_{\Gamma} (\Gamma_0) \cong A$. Moreover, $\Gamma$ is a graded algebra such that $\Gamma_0 \cong A_0^{\textnormal{op}}$ is self-injective as an algebra and generalized Koszul as a $\Gamma$-module. Using this duality, we can exchange $A$ and $\Gamma$ in the above reasoning and get $EE_{\Gamma}(N) \cong N$ for an arbitrary Koszul $\Gamma$-module $N$. Thus $E$ is a dense functor.

Let $L$ be another generalized Koszul $A$-module. Since $L, M, EL, EM$ are all generated in degree 0, we have
\begin{align*}
\hom _{\Gamma} (EL, EM) & \cong \Hom_{\Gamma_0} ((EL)_0, (EM)_0)\\
& \cong \Hom_{\Gamma_0} (\Hom_A (L, A_0), \Hom_A (M, A_0))\\
& \cong \Hom_{A_0 ^{\textnormal{op}}} (\Hom_{A_0} (L_0, A_0), \Hom_{A_0} (M_0, A_0))\\
& \cong \Hom_{A_0} (L_0, M_0) \cong \hom_A (L, M).
\end{align*}
Consequently, $E$ is a duality between the category of generalized Koszul $A$-modules and the category of generalized Koszul $\Gamma$-modules.
\end{proof}

It is well known that a locally finite graded algebra $A$ is classical Koszul if and only if the opposite algebra $A^{\op}$ is also classical Koszul. Unfortunately, this result does not hold in the generalized theory. Here is an example.

\begin{example}
Let $A$ be the path algebra of the following quiver with relations $\delta^2 = \alpha \delta = 0$. Put $A_0 = \langle 1_x, 1_y, \delta \rangle$ and $A_1 = \langle \alpha \rangle$.
\begin{equation*}
\xymatrix{ x \ar@(ld,lu)[]|{\delta} \ar[r]^{\alpha} & y}
\end{equation*}
Then the (graded) left indecomposable projective modules and right indecomposable projective modules are described as follows, where the indices mean the degree.
\begin{equation*}
_LP_x = \begin{matrix} & x_0 & \\ x_0 & & y_1 \end{matrix}, \quad _LP_y = y_0; \quad _RP_x = \begin{matrix} x_0 \\ x_0 \end{matrix}, \quad _RP_y = \begin{matrix} y_0 \\ x_1 \end{matrix}.
\end{equation*}
It is not hard to check that $A$ is a generalized Koszul algebra, but $A^{\op}$ is not. Actually, it is because that $A$ is a projective as a left $A_0$-module, but $A^{\op}$ is not a left projective $A_0^{\op} \cong A_0$-module.
\end{example}

\section{A relation between the generalize theory and the classical theory}

In this section we describe a correspondence between the generalized Koszul theory we just developed and the classical theory. As before, let $A = \bigoplus _{i \geqslant 0} A_i$ be a locally finite graded algebra generated in degrees 0 and 1. At this moment we do \textbf{not} assume the condition that $\fdim A_0 = \fdim A_0^{\op} = 0$.

Let $\mathfrak{r}$ be the radical of $A_0$, and $\mathfrak{R} = A \mathfrak{r} A$ be the two-sided ideal generated by $\mathfrak{r}$. We then define the quotient graded algebra $\bar{A} = A / \mathfrak{R} = \bigoplus _{i \geqslant 0} A_i / \mathfrak{R}_i$. Clearly, $\bar{A}$ is a locally finite graded algebra for which the grading is induced from that of $A$, and $\mathfrak {R} _s = \sum _{i = 0}^s A_i \mathfrak{r} A_{s-i}$. Note that $\bar{A}_0 = A_0 / \mathfrak{r}$ is a semisimple algebra. Given an arbitrary graded $A$-module $M = \bigoplus _{i \geqslant 0} M_i$, define $\bar{M} =  M / \mathfrak{R} M = \bigoplus _{i \geqslant 0} M_i / (\mathfrak{R} M)_i$. Then $\bar{M}$ is a graded $\bar{A}$-module and $\bar{M} \cong \bar{A} \otimes_A M$.

We use an example to show our construction.

\begin{example}
Let $A$ be the path algebra of the following quiver with relations: $\delta^2 = \theta^2 = 0$, $\theta \alpha = \alpha \delta$. Put $A_0 = \langle 1_x, 1_y, \delta, \theta \rangle$ and $A_1 = \langle \alpha, \theta \alpha \rangle$.
\begin{equation*}
\xymatrix{ x \ar@(ld,lu)[]|{\delta} \ar[r]^{\alpha} & y \ar@(rd,ru)[]|{\theta}}
\end{equation*}
The structures of graded indecomposable projective $A$-modules are:
\begin{equation*}
P_x = \begin{matrix} & x_0 & \\ x_0 & & y_1 \\ & y_1 & \end{matrix} \qquad P_y = \begin{matrix} y_0 \\ y_0 \end{matrix}.
\end{equation*}
We find $\mathfrak{r} = \langle \delta, \theta \rangle$, $\mathfrak{R} = \langle \delta, \theta, \theta \alpha \rangle$. Then the quotient algebra $\bar{A}$ is the path algebra of the following quiver with a natural grading:
\begin{equation*}
\xymatrix{ x \ar[r] ^{\alpha} & y}.
\end{equation*}

Let $M = \rad P_x = \langle \delta, \alpha, \alpha \delta \rangle$ which is a graded $A$-module. It has the following structure and is not generated in degree 0:
\begin{equation*}
M = \begin{matrix} x_0 & & y_1 \\ & y_1 & \end{matrix}.
\end{equation*}
Then $\bar{M}_0 = M_0 / \mathfrak{r} M_0 = \langle \bar{\delta} \rangle \cong S_x$, the simple $\bar{A}$-module corresponding to $x$; $\bar{M}_1 = M_1 / (\mathfrak{r} M_1 + A_1 \mathfrak{r} M_0) = \langle \bar{\alpha} \rangle \cong S_y[1]$. Therefore, $\bar{M} \cong S_x \oplus S_y [1]$ is a direct sum of two simple $\bar{A}$-modules, and is not generated in degree 0 either.
\end{example}

The following proposition is crucial to prove the main result in this section.

\begin{proposition}
A graded $A$-module $M$ is generated in degree 0 if and only if the corresponding graded $\bar{A}$-module $\bar{M}$ is generated in degree 0.
\end{proposition}

\begin{proof}
If $M$ is generated in degree 0, then $A_i M_0 = M_i$ for all $i \geqslant 0$. By our construction, it is clear that $\bar{A}_i \bar{M}_0 = \bar{M}_i$. That is, $\bar{M}$ is generated in degree 0.

Conversely, suppose that $\bar{M}$ is generated in degree 0. We want to show $A_i M_0 = M_i$ for $i \geqslant 0$. We use induction to prove this identity. Clearly, it holds for $i = 0$. So we suppose that it is true for all $0 \leqslant i < n$ and consider $M_n$.

Take $v_n \in M_n$ and consider its image $\bar{v}_n$ in $\bar{M}_n = M_n / \sum _{i = 0}^n A_i \mathfrak{r} M_{n-i}$. Since $\bar{M}$ is generated in degree 0, we can find some $a_n \in A_n$ and $v_0 \in M_0$ such that $\bar{v}_n = \bar{a}_n \bar{v}_0$. Thus $\overline {v_n - a_n v_0} = \bar{v}_n - \bar{a}_n \bar{v}_0 = 0$. This means
\begin{equation*}
v_n - a_n v_0 \in \sum_{i = 0}^n A_i \mathfrak{r} M_{n-i} = \mathfrak{r} M_n + \sum _{i=1}^n A_i \mathfrak{r} M_{n-i} = \mathfrak{r} M_n + \sum _{i=1}^n A_i \mathfrak{r} A_{n-i} M_0,
\end{equation*}
where the last identity follows from the induction hypothesis. But it is clear $A_n M_0 \supseteq \sum _{i=1}^n A_i \mathfrak{r} A_{n-i} M_0$, so $v_n - a_n v_0 \in \mathfrak{r} M_n + A_n M_0$. Consequently, $v_n \in \mathfrak{r} M_n + A_n M_0$. Since $v_n \in M_n$ is arbitrary, we have $M_n \subseteq \mathfrak{r} M_n + A_n M_0$. Applying Nakayama's lemma to these $A_0$-modules, we conclude that $M_n = A_n M_0$ as well. The conclusion then follows from induction.
\end{proof}

\begin{lemma}
Let $M$ be a graded $A$-module generated in degree 0. If $P$ is a grade projective cover of $M$, then $\bar{P}$ is a graded projective cover of $\bar{M}$.
\end{lemma}

\begin{proof}
Clearly, $\bar{P}$ is a graded projective module. Both $\bar{P}$ and $\bar{M}$ are generated in degree 0 by the previous proposition. To show that $\bar{P}_0$ is a graded projective cover of $\bar{M}_0$, it suffices to show that $\bar{P}_0$ is a projective cover of $\bar{M}_0$ as $\bar{A}_0$-modules. But this is clearly true since $\bar{P}_0 = P_0 / \mathfrak{r} P_0 \cong M_0 / \mathfrak{r} M_0 \cong \bar{M}_0$.
\end{proof}

In general, $A_1 \mathfrak{r} \neq \mathfrak{r} A_1$. However, if this is true, then $\mathfrak{R}_s = \mathfrak{r} A_s$. Indeed, $\mathfrak {R} _s = \sum _{i = 0}^s A_i \mathfrak{r} A_{s-i}$. Using the fact that $A$ is generated by $A_0$ and $A_1$, and the above commutative relation, we can show $\mathfrak{R}_s = \mathfrak{r} A_i A_{s-i} = \mathfrak{r} A_s$. Therefore, $\bar{A} = \bigoplus _{i \geqslant 0} A_i / \mathfrak{r} A_i$, and for every locally finite graded $A$-module $M = \bigoplus _{i \geqslant 0} M_i$, we have
\begin{equation*}
M_s = M_s / (\mathfrak{R}M)_s = M_s / \sum _{i=0}^s \mathfrak{R}_i M_{s-i} = M_s / \sum_{i=0}^s \mathfrak{r} A_i M_{s-i} = M_s / \mathfrak{r} M_s.
\end{equation*}
Moreover, the procedure of sending $M$ to $\bar{M}$ preserves exact sequences of graded $A$-modules which are projective regarded as $A_0$-modules.

\begin{lemma}
Let $0 \rightarrow L \rightarrow M \rightarrow N \rightarrow 0$ be a short exact sequence of graded $A$-modules such that all terms are projective $A_0$-modules. If $\mathfrak{r}A_1 = A_1 \mathfrak{r}$, then the corresponding sequence $0 \rightarrow \bar{L} \rightarrow \bar{M} \rightarrow \bar{N} \rightarrow 0$ is also exact.
\end{lemma}

\begin{proof}
By the above observation, for $s \geqslant 0$, we have $\bar{L}_s = L_s / \mathfrak{r} L_s$, and similar identities hold for $M$ and $N$. The given exact sequence induces a short exact sequence of $A_0$-modules $0 \rightarrow L_i \rightarrow M_i \rightarrow N_i \rightarrow 0$. Since all terms are projective $A_0$-modules, this sequence splits, and gives a split short exact sequence $0 \rightarrow \mathfrak{r} L_i \rightarrow \mathfrak{r} M_i \rightarrow \mathfrak{r} N_i \rightarrow 0$. Taking quotients, we get an exact sequence of $\bar{A}_0$-modules $0 \rightarrow \bar{L}_i \rightarrow \bar{M}_i \rightarrow \bar{N}_i \rightarrow 0$. Let the index $i$ vary and take direct sum. Then we get an exact sequence of graded $\bar{A}$-modules $0 \rightarrow \bar{L} \rightarrow \bar{M} \rightarrow \bar{N} \rightarrow 0$ as claimed.
\end{proof}

The condition that all terms are projective $A_0$-modules cannot be dropped, as shown by the following example.

\begin{example}
Let $A = A_0 = k[t]/(t^2)$ and $S$ be the simple module and consider a short exact sequence of graded $A$-modules $0 \rightarrow S \rightarrow A \rightarrow S \rightarrow 0$. We have $\bar{A} \cong k$. But the corresponding sequence $0 \rightarrow \bar{S} \rightarrow \bar{A} \rightarrow \bar{S} \rightarrow 0$ is not exact. Actually, the first map $\bar{S} \rightarrow \bar{A}$ is 0 since the image of $S$ is contained in $\mathfrak{r} A_0$.
\end{example}

Now we can prove the main result of this section.

\begin{theorem}
Let $A = \bigoplus _{i \geqslant 1} A_i$ be a locally finite graded algebra and $M$ be a graded $A$-module. Suppose that both $A$ and $M$ are projective $A_0$-modules, and $\mathfrak{r} A_1 = A_1 \mathfrak{r}$. Then $M$ is generalized Koszul if and only if  the corresponding grade $\bar{A}$-module $\bar{M}$ is classical Koszul. In particular, $A$ is a generalized Koszul algebra if and only if $\bar{A}$ is a classical Koszul algebra.
\end{theorem}

\begin{proof}
Let
\begin{equation}
\xymatrix{ \ldots \ar[r] & P^2 \ar[r] & P^1 \ar[r] & P^0 \ar[r] & M \ar[r] & 0}
\end{equation}
be a minimal projective resolution of $M$. Note that all terms in this resolution and all syzygies are projective $A_0$-modules. By Lemmas 3.3 and 3.4, $\bar{M}$ has the following minimal projective resolution
\begin{equation}
\xymatrix{ \ldots \ar[r] & \overline {P^2} \ar[r] & \overline {P^1} \ar[r] & \overline {P^0} \ar[r] & \overline {M} \ar[r] & 0}.
\end{equation}
Moreover, this resolution is linear if and only if the resolution (3.1) is linear by Proposition 3.2. That is, $M$ is generalized Koszul if and only if $\bar{M}$ is classical Koszul. This proves the first statement. Applying it to the graded $A$-module $A_0$ we deduce the second statement immediately.
\end{proof}

If $\fdim A_0 = \fdim A_0^{\op} = 0$, i.e., $A_0$ has the splitting property (S), we have the following corollary:

\begin{corollary}
Let $A = \bigoplus _{i \geqslant 1} A_i$ be a locally finite graded algebra with $\fdim A_0 = \fdim A_0^{\op} = 0$, and suppose $\mathfrak{r} A_1 = A_1 \mathfrak{r}$. Then:
\begin{enumerate}
\item $A$ is a generalized Koszul algebra if and only it is a projective $A_0$-module and $\bar{A}$ is a classical Koszul algebra.
\item Suppose that $A$ is a projective $A_0$-module. A graded $A$-module $M$ is generalized Koszul if and only if it is a projective $A_0$-module and the corresponding grade $\bar{A}$-module $\bar{M}$ is classical Koszul.
\end{enumerate}
\end{corollary}

\begin{proof}
If $A$ is a generalized Koszul algebra, then applying Proposition 2.5 to $_A A$ we conclude that it is a projective $A_0$-module. Moreover, $\bar{A}$ is a classical Koszul algebra by the previous theorem. The converse statement also follows from the previous theorem. This proves the first statement.

If $A$ is a projective $A_0$-module and $M$ is generalized Koszul, by Proposition 2.6 $M$ is a projective $A_0$-module. Moreover, $\bar{M}$ is a classical Koszul module by the previous theorem. The converse statement also follows from the previous theorem.
\end{proof}

We cannot drop the condition that $A$ is a projective $A_0$-module in the above theorem, as shown by the following example.

\begin{example}
Let $A$ be the path algebra of the following quiver with relations: $\delta^2 = \theta^2 = 0$, $\theta \alpha = \alpha \delta = 0$. Put $A_0 = \langle 1_x, 1_y, \delta, \theta \rangle$ and $A_1 = \langle \alpha \rangle$.
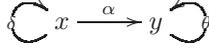
\begin{equation*}
\xymatrix{ x \ar@(ld,lu)[]|{\delta} \ar[r] ^{\alpha} & y \ar@(rd,ru)[]|{\theta}}
\end{equation*}
The structures of graded indecomposable projective $A$-modules are:
\begin{equation*}
P_x = \begin{matrix} & x_0 & \\ x_0 & & y_1 \end{matrix} \qquad P_y = \begin{matrix} y_0 \\ y_0 \end{matrix}.
\end{equation*}
We find $\mathfrak{r} = \langle \delta, \theta \rangle$. Then the quotient algebra $\bar{A}$ is the path algebra of the following quiver:
\begin{equation*}
\xymatrix{ x \ar[r] ^{\alpha} & y}.
\end{equation*}

Let $\Delta_x = P_x / S_y = \langle \delta, 1_x \rangle$ which is a graded $A$-module concentrated in degree 0. The first syzygy $\Omega (\Delta_x) \cong S_y[1]$ is generated in degree 1, but the second syzygy $\Omega^2 (\Delta_x) \cong S_y[1]$ is not generated in degree 2. Therefore, $\Delta_x$ is not generalized Koszul. However, $\bar{\Delta}_x \cong \bar{S}_x$ is obviously a classical Koszul $\bar{A}$-module. Moreover, we can check that $A$ is not a generalized Koszul algebra, but $\bar{A}$ is a classical Koszul algebra.
\end{example}

\section{Applications to directed categories}

In this section we describe some applications of the generalized Koszul theory to certain categories. For the convenience of the reader, let us give some background knowledge.

By the definition in \cite{Li1, Li3}, a \textit{directed category} $\mathcal{A}$ is a $k$-linear category such that there is a partial order $\leqslant$ on $\Ob \mathcal{A}$ satisfying the condition that $x \leqslant y$ whenever $\mathcal{A} (x, y) \neq 0$. In this section all directed categories $\mathcal{A}$ are supposed to have the following conditions: $\mathcal{A}$ is skeletal and has only finitely many objects; $\mathcal{A}$ is \textit{locally finite}, i.e., $\dim _k \mathcal{A}(x,y) < \infty$ for all $x, y \in \Ob \mathcal{A}$; the endomorphism algebra of every object is a local algebra. We also suppose that $\mathcal{A}$ is graded and $\mathcal{A}_0 = \bigoplus _{x \in \Ob \mathcal{A}} \mathcal{A} (x, x)$. Note that the space of all morphisms in $\mathcal{A}$ forms a graded algebra $A$ whose multiplication is determined by composition of morphisms. We call it the \textit{associated algebra} of $\mathcal{A}$.

In \cite{Li1} we described another close relation between the generalized theory and the classical theory for directed categories. The explicit correspondence is described as follows. Let $\mathcal{A}$ be a graded directed category with respect to a partial order $\leqslant$. We define $\mathcal{B}$ to be the graded subcategory of $\mathcal{A}$ formed by replacing the endomorphism algebra of every object by $k \cdot 1$, the span of the identity endomorphism. That is, $\Ob \mathcal{B} = \Ob \mathcal{A}$; $\mathcal{B} (x,y) = \mathcal{A} (x,y)$ if $x \neq y$ and $\mathcal{B} (x, x) = k \cdot 1_x$. Let $A$ and $B$ be the associated graded algebras of $\mathcal{A}$ and $\mathcal{B}$ respectively. Then we have $\bigoplus _{i \geqslant 1} A_i = \bigoplus _{i \geqslant 1} B_i$. Note that $B_0$ is a semisimple algebra, so the classical theory can be applied. On the other hand, $A_0$ as a direct sum of several finite-dimensional local algebras has finitistic dimension 0, so we can use the generalized theory.

\begin{theorem}
Let $A$ and $B$ be defined as above.
\begin{enumerate}
\item Suppose that $A$ is a generalized Koszul algebra. If $M$ is a generalized Koszul $A$-module, then the restricted module $M \downarrow _B^A$ is classical Koszul. In particular, $B$ is a classical Koszul algebra.
\item Suppose that $B$ is a classical Koszul algebra. If $M$ is a graded $A$-module satisfying that $\Omega^i (M)_i$ is a projective $A_0$-module for each $i \geqslant 0$ and $M \downarrow _B^A$ is classical Koszul, then $M$ is generalized Koszul.
\end{enumerate}
\end{theorem}

\begin{proof}
These two statements are precisely Theorems 5.13 and 5.14 in \cite{Li1}. In the original proofs we did not assume that $A_0$ is self-injective, see Remark 5.15.
\end{proof}

\begin{theorem}
Let $A$ and $B$ be defined as above.
\begin{enumerate}
\item $A$ is a generalized Koszul algebra if and only if it is a projective $A_0$-module and $B$ is a classical Koszul algebra.
\item Suppose that $A$ is a generalized Koszul algebra. Then a graded $A$-module $M$ is generalized Koszul if and only if it is a projective $A_0$-module and $M \downarrow _B^A$ is classical Koszul.
\end{enumerate}
\end{theorem}

\begin{proof}
This is Theorem 5.16 in \cite{Li1}, but we drop the unnecessary condition that $A_0$ is self-injective.

(1). If $A$ is a generalized Koszul algebra, then it is a projective $A_0$-module, see the paragraph after Proposition 2.5. By (1) of the previous theorem, $B$ is a classical Koszul algebra. Conversely, if $B$ is a classical Koszul algebra, then $A_0 \downarrow _B^A$ is a classical Koszul $B$-module since it is a projective $B_0$-module. Thus by (2) of the previous theorem, $A$ is a generalized Koszul algebra if we can show that $\Omega^i (A_0)_i$ is a projective $A_0$-module for each $i \geqslant 0$. We prove a stronger statement, that is, $\Omega^i(A_0)$ is a projective $A_0$-module for each $i \geqslant 0$.

Clearly, $\Omega^0 (A_0) = A_0$ is a projective $A_0$-module. Consider the exact sequence
\begin{equation*}
\xymatrix {0 \ar[r] & \Omega^{i+1} (A_0) \ar[r] & P^i \ar[r] & \Omega^i (A_0) \ar[r] & 0}.
\end{equation*}
By the induction hypothesis, $\Omega^i (A_0)$ is a projective $A_0$-module. Thus the above sequence splits as $A_0$-modules. But $P^i$ is a projective $A_0$-module since we assume that $A$ is a projective $A_0$-module, so is $\Omega^{i+1} (A_0)$. This proves (1).

(2). Since $A$ is a generalized Koszul algebra, it is a projective $A_0$-module. If $M$ is generalized Koszul, then it is a projective $A_0$-module (Proposition 2.6) and $M \downarrow _B^A$ is classical Koszul (by (1) of Theorem 4.1). Conversely, if $M \downarrow _B^A$ is classical Koszul, to prove that $M$ is generalized Koszul, by (2) of Theorem 4.1 it suffices to show that $\Omega^i (M)_i$ is a projective $A_0$-module for every $i \geqslant 0$. This can be proved by a similar induction as we just did.
\end{proof}

In practice these two theorems can be used to prove that some complicated directed categories are generalized Koszul. Compared to Theorem 3.6, they have the following disadvantages: They can only be applied to associated algebras $A$ of some directed categories. Moreover, the algebra $B$ is still complicated since we only reduce the sizes of endomorphism algebras of objects in $\mathcal{A}$, and do not remove any non-endomorphisms. Here is an example:

\begin{example}
Let $A$ be the path algebra of the following quiver with relations $\delta^2 = \rho^2 = \theta^2 = 0$, $\alpha \delta = \rho \alpha$, and $\beta \rho = \theta \beta$. Put a grading on $A$ by letting $A_0$ be the space spanned by all endomorphisms and letting $A_1 = \langle \alpha, \beta, \alpha \delta, \beta \rho \rangle$. The reader can check that $A$ is indeed a projective $A_0$-module.
\begin{equation*}
\xymatrix{ x \ar@(ld,lu)[]|{\delta} \ar[r]^{\alpha} & y \ar@(dl,dr)[]|{\rho} \ar[r]^{\beta} & z \ar@(rd,ru)[]|{\theta}}
\end{equation*}

By removing all non-identity endomorphisms, we get an algebra $B$ isomorphic to the path algebra of the following quiver with relations $\beta_2 \alpha_1 = \beta_1 \alpha_2$ and $\beta_2 \alpha_2 = 0$.
\begin{equation*}
\xymatrix{ x \ar@<0.5ex>[r]^{\alpha_1} \ar@<-0.5ex>[r]_{\alpha_2} & y \ar@<0.5ex>[r]^{\beta_1} \ar@<-0.5ex>[r]_{\beta_2} & z}
\end{equation*}

Using the procedure described before Example 3.1, we get a quotient algebra $\bar{A}$ isomorphic to the path algebra of the following quiver:
\begin{equation*}
\xymatrix{ x \ar[r]^{\alpha} & y \ar[r]^{\beta} & z}
\end{equation*}

It is not hard to check that both $B$ and $\bar{A}$ are classical Koszul algebras. Therefore, $A$ is a generalized Koszul algebra either by Theorem 3.6 or by Theorem 4.2. However, the algebra $\bar{A}$ is much simpler.
\end{example}

However, these results can be applied to some algebras $A$ which do not satisfy $\mathfrak{r} A_1 = A_1 \mathfrak{r}$, as shown by the following example.

\begin{example}
Let $A$ be the path algebra of the following quiver with relations $\delta^2 = \gamma \beta \alpha \delta = 0$.
\begin{equation*}
\xymatrix{ x \ar@(ld,lu)[]|{\delta} \ar[r]^{\alpha} & y \ar[r]^{\beta} & z \ar[r]^{\gamma} & w}.
\end{equation*}
This is not a generalized Koszul algebra. Since $\mathfrak{r} = \langle \delta \rangle$, $\mathfrak{R} = \langle \delta, \alpha \delta, \beta \alpha \delta \rangle$, and $\bar{A} = A / \mathfrak{R}$ is the path algebra of a quiver of type $A_4$, which is clearly classical Koszul. Therefore, the result of Theorem 3.6 fails. This is because $A_1 \mathfrak{r} \neq \mathfrak{r} A_1$.

On the other hand, if we use the construction described in this section, we get $B$ is isomorphic to the path algebra of the following quiver with relation $\gamma \beta \alpha' = 0$.
\begin{equation*}
\xymatrix{ x \ar@/^0.5pc/[r]^{\alpha} \ar@/_0.5pc/[r] _{\alpha'} & y \ar[r]^{\beta} & z \ar[r]^{\gamma} & w}.
\end{equation*}
It is easy to check that $B$ is not a classical Koszul algebra. Theorem 4.2 is still true for this example.
\end{example}

We end this section by introducing some possible applications of our theory to certain finite categories which are interesting to people studying finite groups and their representations. These categories are examples of \textit{finite EI categories} (see \cite{Li4, Webb1}, that is, categories with finitely many morphisms such that each endomorphism is an isomorphism.

Let $G$ be a finite group, and $\mathcal{S}$ be a set of subgroups of group $G$. The \textit{transporter category} $\mathcal{T_S}$ has elements in $\mathcal{S}$ as objects. For $H, K \in \mathcal{S}$, $\mathcal{T_S} (H, K) = N_G (H, K) = \{ g \in G \mid (^gH) \subseteq K\}$. Clearly, for each object $H \in \mathcal{S}$, $\mathcal{T_S} (H, H) = N_G(H)$ is the normalizer of $H$ in $G$. Therefore, $\mathcal{T_S}$ is a finite EI category. The \textit{orbit category} $\mathcal{O_S}$ has all cosets $G/H$, where $H \in \mathcal{S}$, as objects, and $\mathcal{O_S} (G/H, G/K)$ is the set of all $G$-equivariant maps. These categories play important roles in group theory. For details, see \cite{Webb2}.

Let $\mathcal{C}$ be a transporter category or an orbit category. By Corollaries 2.6, 2.7 and their proofs in \cite{Webb2}, $\mathcal{C} (y, y)$ acts freely on $\mathcal{C} (x, y)$ for $x, y \in \Ob \mathcal{C}$. Therefore, $\mathcal{C} (x, y)$ is the disjoint union of several orbits, each of which is isomorphic to a copy of $\mathcal{C} (y, y)$. This observation will play a crucial role in our application.

Now we put a grading on all morphisms in $\mathcal{C}$. It would be reasonable to let $\mathcal{C}_0$ be the set of all endomorphisms in $\mathcal{C}$, which is a direct sum of several finite groups. The set of all non-endomorphisms in $\mathcal{C}$ is closed under the action of $\mathcal{C}_0$. Moreover, from \cite{Li4} we know that every non-endomorphism can be factorized as a composite of several \textit{unfactorizable morphisms} (see Definition 2.3 and Proposition 2.5 in \cite{Li4}), and the set of all unfactorizable morphisms is closed under the action of $\mathcal{C}_0$ (Proposition 2.6 in \cite{Li4}). Therefore, we can let $\mathcal{C}_1$ be the set of all unfactorizable morphisms.

It is subtle to define $\mathcal{C}_2$. We cannot define $\mathcal{C}_2$ to be the set of all composites $\beta \alpha$ where $\alpha, \beta \in \mathcal{C}_1$, since it is possible that $\beta \alpha$ can be written as a composite of more than two morphisms in $\mathcal{C}_1$. Therefore, we define $\mathcal{C}_2$ to be the set of morphisms each of which can be expressed as a composite of two unfactorizable morphisms, but cannot be expressed as a composite of more than two morphisms. Similarly, we define $\mathcal{C}_3$ be the set of all morphisms each of which can be expressed as a composite of three unfactorizable morphisms, but cannot be expressed as a composite of more than three morphisms, and so on. Note that since $\mathcal{C} (y, y)$ acts freely on $\mathcal{C} (x, y)$ for $x, y \in \Ob \mathcal{C}$, each $\mathcal{C}_i$ is invariant under the action of $\mathcal{C}_0$.

The category algebra $k\mathcal{C}$ as a vector space has a basis all morphisms in $\mathcal{C}$. The product of two basis elements is their composite (when this is defined in $\mathcal{C}$) or 0 (otherwise). This algebra in general is not graded. However, note that all non-endomorphisms span a two-sided ideal $J$ of $k\mathcal{C}$. Moreover, $k\mathcal{C} / J \cong k\mathcal{C}_0$. Therefore, using this two-sided ideal, we get an \textit{associated graded algebra} $A = \bigoplus _{i \geqslant 0} A_i$ with $A_0 = k\mathcal{C}_0$, a direct sum of group algebras, and $A_i = J^i/J^{i+1}$ for $i \geqslant 1$.

The following proposition tells us that our generalized Koszul theory can apply to the graded algebra $A$ associated to the category algebra $k\mathcal{C}$.

\begin{proposition}
Notation as above, we have:
\begin{enumerate}
\item $\mathcal{C}_i$ is a basis of $A_i$ for each $i \geqslant 0$;
\item $A$ is a projective $A_0$-module;
\item $\bar{A} = A / A \mathfrak{r} A$ is Morita equivalent to a quotient algebra of a finite dimensional hereditary algebra.
\end{enumerate}
\end{proposition}

\begin{proof}
Clearly, for each $i \geqslant 1$, $J^i$ (resp., $J^{i+1})$ is spanned by all morphisms in $\mathcal{C}$ which can be expressed as a composite of $n \geqslant i$ (resp., $n \geqslant i+1$) unfactorizable morphisms. Therefore, the quotient $A_i  = J^i / J^{i+1}$ is spanned by all morphisms which can be expresses as a composite of $i$ unfactorizable morphisms, but can not be expressed as a composite of more than $i$ unfactorizable morphisms. Therefore, $A_i$ is precisely spanned by elements in $\mathcal{C}_i$, which proves (1).

It is obvious that $A_0$ is a projective $A_0$-module. For $i \geqslant 1$, we have a decomposition $\mathcal{C}_i = \sqcup _{x, y \in \Ob \mathcal{C}} \mathcal{C} (x, y)_i$, where $\mathcal{C} (y, y)$ acts freely on $\mathcal{C} (x, y)_i$. Therefore, $k\mathcal{C} (x, y)_i$ is a free $k\mathcal{C} (y, y)$-module, so it is a projective $k\mathcal{C}_0 = A_0$-module. Consequently, $A_i = J^i / J^{i+1} = k\mathcal{C}_i = \oplus _{x, y \in \Ob \mathcal{C}} k\mathcal{C} (x, y)_i$ is a projective $A_0$-module. Thus (2) holds.

Now we prove (3). Take a representative $P_i$ from each isomorphism class of indecomposable projective $k\mathcal{C}$-module, $1 \leqslant i \leqslant n$. Since $\mathcal{C}$ is a finite EI category and these projective modules are pairwise non-isomorphic, there is a partial order $\preccurlyeq$ on $\{ i \} _{i=1}^n$ such that $\Hom _{k\mathcal{C}} (P_i, P_j) = 0$ if $i \npreceq j$. Correspondingly, there is a set of indecomposable projective $A$-modules $\{ Q_i \} _{i=1}^n$, one from each isomorphism class, such that $\Hom _A (Q_i, Q_j) = 0$ if $i \npreceq j$. By taking quotients, $\bar{A}$ has a set of indecomposable projective modules $\{ \bar{Q}_i \} _{i=1}^n$, one from each isomorphism class, satisfying the same condition. Also note that $\End _{\bar{A}} (\bar{Q}_i, \bar{Q}_i) \cong k$ for $1 \leqslant i \leqslant n$. Therefore, the endomorphism algebra of $\oplus _{i=1}^n \bar{Q}_i$ is a quotient algebra of a finite hereditary algebra, which is Morita equivalent to $\bar{A}$.
\end{proof}

As explained in the following example, the correspondence described in Theorem 3.6 indeed gives us a feasible way to show a complicated category or algebra to be generalized Koszul.

\begin{example}
Let $G = C_6$ be a cyclic group of order 6, and let $\mathcal{S}$ be the set of 4 subgroups: $x = 1$, $w = C_6$, $y = C_2$, and $z = C_3$. The structure of the transporter category $\mathcal{T_S}$ is pictured as below:
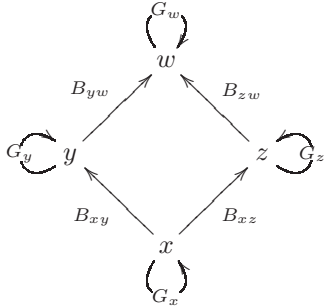
\begin{equation*}
\xymatrix{ & w \ar@(ul,ur)[]|{G_w} & \\
y \ar@(ld,lu)[]|{G_y} \ar[ur] ^{B_{yw}} & & z \ar@(rd,ru)[]|{G_z} \ar[ul] _{B_{zw}} \\
 & x \ar@(dl,dr)[]|{G_x} \ar[ur] _{B_{xz}} \ar[ul] ^{B_{xy}}}
\end{equation*}
By computation, we find $G_x \cong G_z \cong G_w \cong G_y \cong C_6$. Moreover, the four bisets $B_{xz} \cong B_{xy} \cong B_{yw} \cong B_{zw} \cong C_6$. Note that we also have $B_{yw} B_{xy} = B_{zw} B_{xz}$ since they both equal the set of morphisms from $x$ to $w$, which is also isomorphic to a copy of $C_6$. This category can be graded in an obvious way. Let $A$ be the graded category algebra. We find $A_1\mathfrak{r} = \mathfrak{r} A_1$.

The $k$-linear representations of this category (which are the same as $k\mathcal{T_S}$-modules) depends on the characteristic of $k$. Let us do the computation for an algebraically closed field of characteristic $3$. Then $kC_6$ is the direct sum of two non-isomorphic indecomposable projective modules. An explicit computation tells us that $\bar{A}$ is a direct product of two graded algebras, each being isomorphic to the graded incidence algebra of the following poset with length grading:
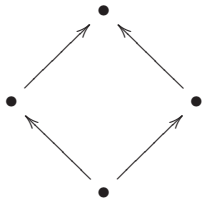
\begin{equation*}
\xymatrix{ & \bullet & \\
\bullet \ar[ur] & & \bullet \ar[ul] \\
 & \bullet \ar[ur] \ar[ul]}
\end{equation*}
Since this graded incidence algebra is classical Koszul, by Theorem 3.6, the category algebra $k\mathcal{C}$ is generalized Koszul.
\end{example}

However, we do not know for what transporter categories or orbit categories $\mathcal{C}$, the category algebras $k\mathcal{C}$ or the associated graded algebras $A$ are generalized Koszul.

\section{Applications to standardly stratified algebras}

Let $A$ be a basic finite-dimensional algebra whose simple modules $S_{\lambda}$ (up to isomorphism) are indexed by a preordered set $(\Lambda, \leqslant)$. This preordered set also indexes all indecomposable projective $A$-modules $P_{\lambda}$ up to isomorphism. According to \cite{Cline}, $A$ is \textit{standardly stratified} with respect to $\leqslant$ if there exist modules $\Delta_{\lambda}$, $\lambda \in \Lambda$, such that the following conditions hold:
\begin{enumerate}
\item the composition factor multiplicity $[\Delta_{\lambda} : S_{\mu}] = 0$ whenever $\mu \nleqslant \lambda$;
and
\item for every $\lambda \in \Lambda$ there is a short exact sequence $0 \rightarrow K_{\lambda} \rightarrow P_{\lambda} \rightarrow \Delta_{\lambda} \rightarrow 0$ such that $K_{\lambda}$ has a filtration with factors $\Delta_{\mu}$ where $\mu > \lambda$.
\end{enumerate}
The \textit{standard module} $\Delta_{\lambda}$ is the largest quotient of $P_{\lambda}$ with only composition factors $S_{\mu}$ such that $\mu \leqslant \lambda$. It has the following description:
\begin{equation*}
\Delta_{\lambda} = P_{\lambda} / \sum _{\mu \nleqslant \lambda} \text{tr} _{P_{\mu}} (P_{\lambda}),
\end{equation*}
where tr$_{P_{\mu}} (P_{\lambda})$ is the trace of $P_{\mu}$ in $P _{\lambda}$ (\cite{Dlab, Webb2}).

Throughout this section we suppose that $A$ is standardly stratified with respect to a partial order $\leqslant$. Let $\Delta$ be the direct sum of all standard modules and $\mathcal{F} (\Delta)$ be the full subcategory of $A$-mod such that each object in $\mathcal{F} (\Delta)$ has a filtration by standard modules. For $M \in \mathcal{F} (\Delta)$ and $\lambda \in \Lambda$, we take a particular $\Delta$-filtration $\xi$ and define the multiplicity $[M: \Delta_{\lambda}]$ to be the number of factors in $\xi$ isomorphic to $\Delta_{\lambda}$. It is well known that the multiplicity is independent of the choice of a particular filtration.

Since standard modules are relative simple in the category $\mathcal{F} (\Delta)$ and have finite projective dimensions, the extension algebra $\Gamma = \Ext _A^{\ast} (\Delta, \Delta)$ of standard modules is a graded finite-dimensional algebra, and provides us a lot information on the structures of indecomposable objects in $\mathcal{F} (\Delta)$. The structure of $\Gamma$ has been considered in \cite{Abe, Agoston1, Agoston2, Drozd, Klamt, Li2, Madsen3, Mazorchuk1, Mazorchuk3, Mazorchuk4, Miemietz}. However, most authors assumed that $A$ has a grading such that $A_0$ is semisimple, and this grading is compatible with the $\Delta$-filtration. That is, the top of each standard module and the socle of each costandard module lie in degree 0. Under this hypothesis, and using some extra assumptions, they show that the Koszul dual algebra is also standardly stratified, the Ringel dual algebra is also Koszul if so is $A$, and the Ringel duality and the Koszul duality commute.

In this paper we study the extension algebras of standard modules from a different approach. In general we do not assume that the standardly stratified algebra is graded. Indeed, since the extension algebra has a natural grading, and its degree 0 part, which is exactly the endomorphism algebra of standard modules, is in general not semisimple, our generalized Koszul theory can apply to this situation. Actually, if the $\Delta$-filtration of indecomposable projective modules has nice combinatorial property, we can still get certain Koszul property for the extension algebra.

We introduce some notations. Let $\Lambda_1$ be the subset of all minimal elements in $\Lambda$, $\Lambda_2$ be the subset of all minimal elements in $\Lambda \setminus \Lambda_1$, and so on. Then $\Lambda = \sqcup_{i \geqslant 1} \Lambda_i$. With this partition, we can introduce a \textit{height function} $h: \Lambda \rightarrow \mathbb{N}$ in the following way: for $\lambda \in \Lambda_i \subseteq \Lambda$, $i \geqslant 1$, we define $h(\lambda) = i$. For each $M \in \mathcal{F} (\Delta)$, we define supp$(M)$ to be the set of elements $\lambda \in \Lambda$ such that $[M: \Delta_{\lambda}] \neq 0$. For example, supp$ (\Delta_{\lambda}) = \{ \lambda \}$. We also define $\min (M) = \min (\{ h(\lambda) \mid \lambda \in \text{supp} (M) \})$. We say $M$ is \textit{generated in height $i$} if every simple summand of $M / \rad M$ is isomorphic to some $S_{\lambda}$ with $h (\lambda) = i$. For example, the standard module $\Delta_{\lambda}$ is generated in height $h(\lambda)$. Let $\Gamma = \Ext_A^{\ast} (\Delta, \Delta)$.

The following definition is an analogue of generalized Koszul modules of graded algebras.

\begin{definition}
An $A$-module $M \in \mathcal{F} (\Delta)$ is said to be linearly filtered if $M$ is generated in certain height $i \geqslant 0$ and has a projective resolution
\begin{equation*}
\xymatrix{ 0 \ar[r] & Q^l \ar[r] & Q^{l-1} \ldots \ar[r] & Q^{i+1} \ar[r] & Q^i \ar[r] & M \ar[r] & 0}
\end{equation*}
such that each $Q^s$ is generated in height $s$, $i \leqslant s \leqslant l$.
\end{definition}

With this terminology, we have:

\begin{theorem}
Suppose that $\Delta \cong \Gamma_0$ as a $\Gamma_0$-module and $\Delta_{\lambda}$ is linearly filtered for each $\lambda \in \Lambda$. Then:
\begin{enumerate}
\item $\Gamma$ is a projective $\Gamma_0$-module.
\item If $M \in \mathcal{F} (\Delta)$ is linearly filtered, then $\Ext _A^{\ast} (M, \Delta)$ is a generalized Koszul $\Gamma$-module.
\item In particular, $\Gamma$ is a generalized Koszul algebra.
\end{enumerate}
\end{theorem}

\begin{proof}
The last two statements come from Theorem 2.12 in \cite{Li2} by letting $\Theta = \Delta$ and $Q = _AA$. Thus we only need to show the first statement. But by (2.2) in page 14 of \cite{Li2}, we know that $\Gamma_s = \bigoplus _{\lambda \in \Lambda} \Ext _A^s (\Delta_{\lambda}, \Delta)$ is a projective $\Gamma_0$-module for every $s \geqslant 1$.
\end{proof}

We remind the reader that although $\Gamma$ is a generalized Koszul algebra, the Koszul duality in general does not hold since $\fdim \Gamma_0$ might be nonzero. See Example 2.14 in \cite{Li2}.

As before, let $\mathfrak{r} = \rad \Gamma_0$ and $\bar{\Gamma} = \Gamma / \Gamma \mathfrak{r} \Gamma$. Then we have an immediate corollary.

\begin{corollary}
Suppose that $\Delta \cong \Gamma_0$ as a $\Gamma_0$-module and $\Delta_{\lambda}$ is linearly filtered for each $\lambda \in \Lambda$. If $\Gamma_1 \mathfrak{r} = \mathfrak{r} \Gamma_1$, then the quotient algebra $\bar{\Gamma}$ is a classical Koszul algebra.
\end{corollary}

\begin{proof}
The conclusion follows from Theorems 5.2 and 3.6.
\end{proof}

\begin{example}
Let $A$ be the path algebra of the following quiver with relations $\delta^2 = \delta \alpha = \beta \delta = \beta \alpha = \gamma \beta =0$. Let $x > z > y$.
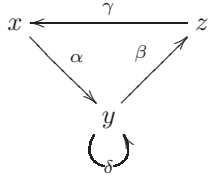
\begin{equation*}
\xymatrix{ x \ar[dr] ^{\alpha} & & z \ar[ll] _{\gamma}\\
& y \ar[ur] ^{\beta} \ar@(dl,dr)[]|{\delta} }
\end{equation*}
Indecomposable projective modules and standard modules of $A$ are described below:
\begin{align*}
& P_x = \begin{matrix} x \\ y \end{matrix} \quad P_y = \begin{matrix} & y & \\ y & & z \end{matrix} \quad P_z = \begin{matrix} z \\ x \\ y \end{matrix} \\
& \Delta_x = P_x = \begin{matrix} x \\ y \end{matrix} \quad \Delta_y = \begin{matrix} y \\ y \end{matrix} \quad \Delta_z = z.
\end{align*}
Clearly, $A$ is standardly stratified. Moreover, all standard modules have projective dimension 1 and are linearly filtered. By direct computation we check that $\Delta \cong \End _A (\Delta)$ as $\Gamma_0 = \End _A (\Delta)$-modules.

Now we compute the extension algebra $\Gamma$: $\Gamma_s = 0$ for $s \geqslant 2$; $\Ext _A^1 (\Delta_x, \Delta) = 0$; $\Ext _A^1 (\Delta_y, \Delta) \cong \End_A (\Delta_z)$; and $\Ext _A^1 (\Delta_z, \Delta) \cong \End_A (\Delta_x)$. Therefore, we find $\Gamma$ is the path algebra of the following quiver with relations $\delta^2 = \beta \delta = \alpha \delta = \gamma \beta =0$.
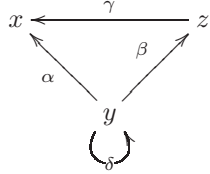
\begin{equation*}
\xymatrix{ x & & z \ar[ll] _{\gamma}\\
& y \ar[ur] ^{\beta} \ar[ul] ^{\alpha} \ar@(dl,dr)[]|{\delta} }
\end{equation*}
\end{example}

We remind the reader that $\alpha$ is in the degree 0 part of $\Gamma$. Indeed, $\Gamma_0 = \langle 1_x, 1_y, 1_z, \delta, \alpha \rangle$ and $\Gamma_1 = \langle \beta, \gamma \rangle$. In this case $\fdim \Gamma_0 \neq 0$ since there is a non-split exact sequence:
\begin{equation*}
0 \rightarrow  x \rightarrow \begin{matrix} & y & \\ y & & x \end{matrix} \rightarrow \begin{matrix} y \\ y \end{matrix} \rightarrow 0.
\end{equation*}

Since $\mathfrak{r} = \rad \Gamma_0 = \langle \delta, \alpha \rangle$, and $\mathfrak{r} \Gamma_1 = \Gamma_1 \mathfrak{r} = 0$, the quotient algebra $\Gamma$ is the path algebra of the following quiver with relation $\gamma \beta = 0$, which is clearly a classical Koszul algebra.
\begin{equation*}
\xymatrix {x & z \ar[l]_{\gamma} & y \ar[l] _{\beta}}.
\end{equation*}


\begin{thebibliography}{99}
\bibitem{Abe} N. Abe, \textit{First extension groups of Verma modules and $R$-polynomials}, arXiv:1003.0169.
\bibitem{Agoston1} I. \'{A}goston, V. Dlab, and E. Luk\'{a}s, \textit{Quasi-hereditary extension algebras}, Algebras and Representation Theorey 6 (2003), 97-117.
\bibitem{Agoston2} I. \'{A}goston, V. Dlab, and E. Luk\'{a}s, \textit{Standardly stratified extension algebras}, Comm. Algebra 33 (2005), 1357-1368.
\bibitem{Bass} H. Bass, \textit{Finitistic dimension and a homological generalization of semi-primary rings}, Trans. AMS, 95 (1961) 466-488.
\bibitem{BGS} A. Beilinson, V. Ginzburg, and W. Soergel, \textit{Kosuzl duality patterns in representation theory}, J. Amer. Math. Soc. 9 (1996), 473-527.
\bibitem{Cline} E. Cline, B. Parshall, and L. Scott, \textit{Stratifying endomorphism algebras}, Mem. Amer. Math. Sco. 124 (1996), no. 591.
\bibitem{Dlab} V. Dlab, and C. Ringel, \textit{The module theoretical approach to Quasi-hereditary Algebras}, Representations of algebras and related topics (Kyoto, 1990), 200-224, London Math. Soc. Lecture Note Ser. 168, Cambridge Univ. Press 1992.
\bibitem{Drozd} Y. Drozd, V. Mazorchuk, \textit{Koszul duality for extension algebras of standard modules}, J. Pure Appl. Algebra 211 (2007), 484-496.
\bibitem{Green1} E. L. Green and R. Mart\'{\i}nez-Villa, \textit{Koszul and Yoneda algebras}, Representation theory of algebras (Cocoyoc, 1994), 247-297, CMS Conf. Proc., 18, Amer. Math. Soc., Providence, RI, 1996.
\bibitem{Green2} E. L. Green and R. Mart\'{\i}nez-Villa, \textit{Koszul and Yoneda algebras II: algebras and modules II}, Geiranger, 1996, 227-244, CMS Conf. Proc., 24, Amer. Math. Soc., Providence, RI, 1998.
\bibitem{Green3} E. L. Green, I. Reiten, and {\O}. Solberg, \textit{Dualities on generalized Koszul algebras}, Mem. Amer. Math. Soc. 159 (2002), xvi+67pp.
\bibitem{Jans} J. P. Jans, \textit{Duality in noetherian rings}, Proc. AMS, 12 (1961), 829-835.
\bibitem{Klamt} A. Klamt, C. Stroppel, \textit{On the Ext algebras of parabolic Verma modules and $A_{\infty}$-structures}, J. Pure Appl. Algebra 216 (2012), 323-336.
\bibitem{Li4} L. Li, \textit{A Chracterization of Finite EI Categories with Hereditary Category Algebras}, J. Algebra 345 (2011), 213-241.
\bibitem{Li2} L. Li, \textit{Extension algebras of standard modules}, Comm. Algebra 41 (2013), 3445-3464.
\bibitem{Li3} L. Li, \textit{Algebras stratified for all linear orders},  Alg. Rep. Theory 16 (2013), 1085-1108.
\bibitem{Li1} L. Li, \textit{A generalized Koszul theory and its application}, Trans. Amer. Math. Soc. 366 (2014), 931-977.
\bibitem{Madsen1} D. Madsen, \textit{Ext-algebras and derived equivalences}, Colloq. Math. 104 (2006), 113-140.
\bibitem{Madsen2} D. Madsen, \textit{On a common generalization of Koszul duality and tilting equivalence}, Adv. Math. 227 (2011), 2327-2348.
\bibitem{Madsen3} D. Madsen, \textit{Quasi-hereditary algebras and generalized Koszul duality}, preprint, arXiv:1201.0441.
\bibitem{Martinez} R. Mart\'{\i}nez-Villa, \textit{Introduction to Koszul algebras}, Rev. Un. Mat. Argentina 48 (2007), 67-95.
\bibitem{Mazorchuk1} V. Mazorchuk, \textit{Some homological properties of the category $\mathcal{O}$}, Pacific J. Math. 232 (2007), 313-341.
\bibitem{Mazorchuk2} V. Mazorchuk, S. Ovsienko, and C. Stroppel, \textit{Quadratic duals, Koszul dual functors, and applications}, Trans. Amer. Math. Soc. 361 (2009), 1129-1172.
\bibitem{Mazorchuk3} V. Mazorchuk, \textit{Koszul duality for stratified algebras I: balanced quasi-hereditary algebras}, Manuscripta Math. 131 (2010), 1-10.
\bibitem{Mazorchuk4} V.Mazorchuk, \textit{Koszul duality for stratified algebras II: standardly stratified algebras}, J. Aust. Math. Soc. 89 (2010), 23-49.
\bibitem{Miemietz} V. Miemietz, W. Turner, \textit{The Weyl extension algebra of $GL_2(\bar{\mathbb{F}}_p)$}, arXiv: 1106.5665.
\bibitem{Webb1}P. Webb, \textit{An introduction to the representations and cohomology of categories}, Group Representation Theory, M. Geck, D. Testerman and J. Th¨¦venaz(eds), EPFL Press (Lausanne), 149-173, (2007).
\bibitem{Webb2} P. Webb, \textit{Standard stratifications of EI categories and Alperin's weight conjecture}, J. Algebra 320 (2008), 4073-4091.
\bibitem{Woodcock} D. Woodcock, \textit{Cohen-Macaulay complexes and Koszul rings}, J. London Math. Soc. (2) 57 (1998), 398-410.
\end{thebibliography}
\end{document}